\documentclass[12pt,leqno,twoside]{amsart}
\usepackage{amssymb}
\usepackage{latexsym}
\usepackage{graphicx}
\usepackage{epsfig}
\usepackage{srcltx}

\newtheorem{theorem}{Theorem}[section]
\newtheorem{corollary}[theorem]{Corollary}
\newtheorem{lemma}[theorem]{Lemma}
\newtheorem{proposition}[theorem]{Proposition}

\theoremstyle{remark}
\newtheorem{remark}[theorem]{\sc Remark}
\theoremstyle{remark}

\theoremstyle{definition}
\newtheorem{definition}[theorem]{Definition}
\theoremstyle{remark}
\newtheorem{example}[theorem]{\sc Example}

\theoremstyle{remark}

\theoremstyle{remark}

\numberwithin{equation}{section}  

\renewcommand{\Box}{\square}    

\newcommand{\cal}{\mathcal}

\newcommand{\rank}{\mathrm{rank\hspace{2pt}}}

\renewcommand{\int}{{\mathrm{int}}}

\newcommand{\Sing}{{\mathrm{Sing\hspace{2pt}}}}

\newcommand{\id}{{\mathrm{id}}}

\newcommand{\im}{{\mathrm{Im\hspace{1pt}}}}

\renewcommand{\ker}{\mathop{{\mathrm{ker}}}\nolimits}
\newcommand{\coker}{\mathop{{\mathrm{coker}}}\nolimits}

\newcommand{\e}{\varepsilon}
\newcommand{\m}{\setminus}
\newcommand{\fin}{\hspace*{\fill}$\Box$\vspace*{2mm}}

\newcommand{\tF}{F^{\pitchfork}}
\newcommand{\tE}{E^{\pitchfork}}
\newcommand{\tmu}{\mu^{\pitchfork}}

\newcommand{\cA}{{\cal A}}
\newcommand{\cB}{{\cal B}}

\newcommand{\cC}{{\cal C}}

\newcommand{\cO}{{\cal O}}

\newcommand{\cX}{{\cal X}}
\newcommand{\cN}{{\cal N}}
\newcommand{\cY}{{\cal Y}}
\newcommand{\cZ}{{\cal Z}}

\newcommand{\bC}{{\mathbb C}}
\newcommand{\bD}{{\mathbb D}}

\newcommand{\bP}{{\mathbb P}}

\newcommand{\bL}{{\mathbb L}}

\newcommand{\bZ}{{\mathbb Z}}

\newcommand{\bQ}{{\mathbb Q}}

\newcommand{\bV}{{\mathbb V}}

\begin{document}

\title[Vanishing homology of projective hypersurfaces]
 {Vanishing homology of projective hypersurfaces
with 1-dimensional singularities}

\author{\sc Dirk Siersma}  

\address{Institute of Mathematics, Utrecht University, PO
Box 80010, \ 3508 TA Utrecht, The Netherlands.}

\email{D.Siersma@uu.nl}

\author{\sc Mihai Tib\u ar}

\address{Univ. Lille, CNRS, UMR 8524 - Laboratoire Paul Painlev\'e, F-59000 Lille, France.}

\email{tibar@math.univ-lille1.fr}

\thanks{The authors express their gratitude to CIRM in Luminy for supporting this research project through the ``Recherche en Bin\^{o}me'' program. The second named author acknowledges the support of the Labex CEMPI (ANR-11-LABX-0007-01).}




\dedicatory{}



\begin{abstract}  

We introduce and study the vanishing homology of singular projective hypersurfaces.  We prove its concentration in two levels in case of 1-dimensional singular locus $\Sigma$, and moreover determine the ranks of the nontrivial homology groups. These two groups depend on the monodromy  at special points of   $\Sigma$ and on the effect of the monodromy of the local system over its complement.
\end{abstract}
 
\maketitle

\setcounter{section}{0}

\section{Introduction and results}\label{s:intro}

 The homology of a projective hypersurface $V\subset \bP^{n+1}$ is known for smooth $V$ whereas only few results are available in the singular setting. The classical Lefschetz Hyperplane Theorem (LHT) yields that the inclusion of spaces induces an isomorphism:
\begin{equation}\label{eq:lht}
 H_k(V, \bZ) \stackrel{\simeq}{\to}  H_k(\bP^{n+1},\bZ)
\end{equation}
 for  $k < n$ and an epimorphism for $k=n$, independently on the singular locus $\Sing V$. 
Since $V$ is a CW-complex of dimension $2n$, the remaining task is to find the homology groups $H_k(V, \bZ)$ for $k\ge n$.

\smallskip
  
In case of a smooth hypersurface $V_{n,d}$ all homology groups appear to be free and by Poincar\'e and Lefschetz duality\footnote{see  \cite{Di0} for details}:
 $H_k(V_{n,d}, \bZ) \cong H_k(\bP^{n},\bZ)$ if $k \ne n$ and the rank of $H_n(V_{n,d},\bZ)$
follows from the Euler characteristic computation
$\chi(V_{n,d})  =  n+2 - \frac{1}{d} [ 1 + (-1)^{n+1} (d-1)^{n+2} ]$. Smooth projective complete intersections have been studied by Libgober and Wood \cite{LW}.

In the 1980's Dimca studied the case of isolated singularities \cite{Di0}, \cite{Di}; we shall discuss his main result \cite[Thm. 4.3]{Di} in \S \ref{s:vh-is}.

\smallskip

Our paper focuses on the first unknown case, $\dim \Sing V= 1$. We approach the singular hypersurface $V$ by 
comparing its integer homology to  that of a smooth hypersurface of the same degree, as an intermediate step towards computing the homology of singular hypersurfaces. A different viewpoint, based on Griffiths cohomological techniques,  has been taken by Hulek and Kloosterman in the  study of elliptic 3-folds \cite{HK}. 

\smallskip 

 We therefore introduce and study the ``vanishing homology" of $V$, as follows.  
\begin{definition}
Let $f=0$ be the defining equation of $V \subset \bP^{n+1}$ as a reduced hypersurface, where $d= \deg f$. Consider the following one-parameter smoothing of degree $d$,  $V_\e := \{ f_\e = f + \e h_d = 0\}$, where $h_d$ denotes a general homogeneous polynomial of degree $d$.   Let:
\[  \bV_\Delta := \{ (x, \e) \in \bP^{n+1} \times \Delta \mid f + \e h_d = 0\} 
\]
denote the total space of the pencil, where $V_0 := V \subset \bP^{n+1} \times \{ 0\}$ and $\Delta$ is a small enough disk centered at $0\in \bC$ such that $V_\e$ is non-singular for all $\e \in \Delta^*$. Let $A := \{ f= h_d = 0\}$ be the axis of the pencil and let
$\pi : \bV_\Delta \to \Delta$ denote the projection. We define:
\[  H^{\curlyvee}_{*}(V) := H_*(\bV_\Delta, V_\e;\bZ)\]
and call it the \emph{vanishing homology of $V$}. 
\end{definition}

The genericity of $h_d$ ensures the existence of small enough disks $\Delta$ as in the above definition, see e.g. \cite[Prop. 2.2]{ST-bettimax}.
Note that $\bV_\Delta$ retracts to $V$, thus the vanishing homology compares $V$ to the smooth hypersurface $V_\e$ of the same degree. 
Since all smooth hypersurfaces of fixed degree are homeomorphic,  the vanishing homology does not depend on the particular smoothing of degree $d$,  it is thus an invariant of $V$.

With the vanishing homology we recover Dimca's result for isolated singularities \cite[Thm. 4.3]{Di}, see Proposition \ref{p:dimca} and Proposition \ref{p:dimca1}.

\medskip

Our first result Theorem \ref{t:main} is that the vanishing homology $H^{\curlyvee}_{*}(V)$, in case $\dim \Sing V =1$, is concentrated in dimensions $n+1$ and $n+2$ only. 


By the exact sequence of the pair $(\bV_\Delta, V_\e)$, the concentration of the vanishing homology implies the isomorphisms:

\[H_k(V,\bZ) \simeq  H_k(V_{n,d},\bZ) \simeq  H_k(\bP^n,\bZ) \ \mbox{for} \ k \not= n, n+1,n+2.\] 

In the second part of the paper  we investigate the relations among the remaining homology groups $H_k(V,\bZ)$ (i.e. $k= n, n+1, n+2$)  and the vanishing homology,  and we single out remarkable particular cases.

\smallskip

  Our main results in \S \ref{s:Betti} are
formulas for the (ranks of the) possibly non-trivial groups $H^{\curlyvee}_{n+1}(V)$ and $H^{\curlyvee}_{n+2}(V)$. They depend on the information about local isolated or special non-isolated singularities, the properties of the curve part of $\Sing V$, the transversal singularity types and the monodromies along loops in the transversal local systems. The singular locus $\Sing V$ has a finite set $R$ of isolated points and finitely many curve branches. Each such branch $\Sigma_i$ of $\Sing V$ has a generic transversal type (of transversal Milnor fibre $\tF_i$ and Milnor number denoted by $\tmu_i$) and the axis $A$ cuts it at a finite set of general points $P_i$. It also contains a finite set $Q_i$ of points with non-generic transversal type, which we call \emph{special points}, and we denote by $\cA_q$ the local Milnor fibre at $q\in Q$. At each point $q\in Q_i$ there are finitely many locally irreducible branches of the germ $(\Sigma_i, q)$,  we denote by $\gamma_{i,q}$ their number and let $\gamma_{i} := \sum_{i\in Q_i}\gamma_{i,q}$ (see \S \ref{ss:notations} for the notations). 

\smallskip

Our Theorem \ref{t:local_fibres}
determines $H^{\curlyvee}_{n+2}(V)$ as an intersection of local and global contributions in a reference space consisting of the direct sum of the homology of the transversal fibres. As a consequence it 
 tells that the $(n+2)$th vanishing Betti number is bounded by the sum of all Milnor numbers of transversal singularities, taken over all irreducible 1-dimensional components of $\Sing V$, and each special singular point on $\Sing V$ with non-trivial transversal monodromy decreases this Betti-number.

\smallskip

Corollary \ref{c:annul} (see also Example \ref{ex:red}) tells that
if for each irreducible 1-dimensional component $\Sigma_i$ of $\Sing V$ we have at least one local special singularity with rank zero $(n-1)$th homology group, then the vanishing homology of $V$  is free, concentrated in dimension $n+1$ only, and the corresponding Betti number is given by the formula:
\[  b_{n+1}(\bV_\Delta, V_\e) =\sum_i (\nu_i+ \gamma_i + 2g_i -2)\tmu_i +
(-1)^n \sum_{q \in Q} (\chi(\cA_q)-1)
+ \sum_{r \in R} \mu_r , 
\]
where $Q := \cup_i Q_i$, $\nu_i := \# P_i$, $\mu_r$ is the Milnor number of the isolated singularity germ $(V, r)$, and $g_i$ is the genus of $\Sigma_i$ (see \S \ref{ss:cw2} for the meaning of the genus in case of singular $\Sigma_i$).

\smallskip

 In our proofs we use in particular the detailed construction of a CW-complex model of the pair $(\bV_\Delta, V_\e)$ which is done in \S \ref{ss:cw}  and \S \ref{ss:cw2}. We also use the full strength of the results on local 1-dimensional singularities found by Siersma \cite{Si1}, \cite{Si2}, \cite{Si3}, \cite{Si-Cam}, cf also \cite{Yo}, \cite{Ti-nonisol}, which involve the study of the local system of transversal Milnor fibres. 

\smallskip

We provide several examples in \S \ref{s:remarks}. In certain cases we  can prove  the freeness of the $(n+1)$th vanishing homology group. We also show an example where the homology of $V$ over $\bQ$ may be computed via our formulas for the vanishing homology.

\smallskip
Let us finally mention a couple of recent applications. 
 The 1-dimensional locus case appeared recently in work of  Fr\"uhbis-Kr\" uger and Zach \cite{FZ}, \cite{Za}.
 They have studied, following work by Damon and Pike \cite{DP}, the  vanishing cycles of a certain class of smoothable isolated
Cohen-Macaulay codimension 2 singularities. As Tjurina transforms yield non-isolated singularities which can be studied with the methods of our paper, they could obtain in this way more detailed insight over the vanishing topology of a certain class of isolated determinantal singularities. 
Also recently we have computed  the homology of a local Milnor fibre  via admissible deformations  \cite{ST-MiFiH} by using the approach of this paper.

\section{Vanishing homology in case of isolated singularities} \label{s:vh-is}

Throughout this paper we use homology over $\bZ$ unless otherwise stated.
Let $V := \{ f= 0 \} \subset \bP^{n+1}$ be a hypersurface of degree $d$ with singular locus consisting of a finite set  of points $R$.
 Since $V$ has only isolated singularities, the genericity of the axis $A = \{ f=h_d=0\}$  of the pencil $\pi : \bV_\Delta \to \Delta$ just means that $A$ avoids $R$. It turns out (see also \cite[\S 5]{ST-bettimax}) that $\bV_\Delta$ is non-singular and that the projection $\pi$ has isolated singularities 
precisely at the points of $R$.  Given some ball $B\subset \bP^{n+1} \times \Delta$, we shall denote
the intersection $B\cap \bV_\Delta$ simply by $B$, for the sake of simplicity. 

 For small enough balls $B_r$, at each point $r\in R$, the 
 homotopy retraction within the fibration $\pi$ yields the isomorphism:
\begin{equation}\label{eq:latice}
  H_*(\bV_\Delta, V_\e) \simeq \bigoplus_{r \in R} H_*(B_r, B_r \cap V_\e)
\end{equation}
where 
 $B_r \cap V_\e$ is the Milnor fibre of the isolated hypersurface singularity germ $(V, r)$. The relative homology $H_*(B_r, B_r \cap V_\e)$
is concentrated in dimension $n+1$ and $H_{n+1}(B_r, B_r \cap V_\e)$ is isomorphic to the Milnor lattice $\bL_r$ of the hypersurface germ $(V,r)$, thus  isomorphic to $\bZ^{\mu_r}$, where $\mu_r$ is the Milnor number of $(V, r)$. We get the following conclusion:
\begin{lemma} \label{p:concIsol}
If $\dim \Sing V \le 0$ then:
\[ 
H^{\curlyvee}_{k}(V) = 0 \; \mbox{\rm if} \; k \ne n+1, \]
 \[ H^{\curlyvee}_{n+1}(V) = \bigoplus_{r\in R} \bL_r .
 \]\fin
\end{lemma}
From the long 
exact sequence of the pair $(\bV_\Delta, V_\e)$ we also obtain the 5-terms exact sequence:
$$ 0 \rightarrow  H_{n+1}(V_\e) \rightarrow H_{n+1}(V) \rightarrow \bigoplus_{r\in R} \bL_r  \xrightarrow{\Phi_{n}}  \bL \rightarrow H_{n}(V) \rightarrow 0 $$
where $\bL :=  H_n(V_\e)$ is the intersection lattice of the middle homology of the smooth hypersurface of degree $d$ and the map $\Phi_n$ is identified to the boundary map $H_{n+1}(\bV_\Delta, V_\e) \to H_{n}(V_\e)$.   
 We get the integer homology of $V$ as follows:
\begin{proposition}\label{p:dimca}
\begin{enumerate}
 \item $H_k(V) \simeq H_{k}(\bP^n) $  for $k\not= n, n+1$,
\item $H_{n+1}(V)  \simeq H_{n+1}(\bP^n) \bigoplus\ker \Phi_{n}$, 
\item $H_{n}(V)  \simeq  \coker \Phi_{n}$.
\fin
\end{enumerate}
\end{proposition}

This is striking similar to Dimca's  result  \cite[Theorem 2.1]{Di0}, \cite[Theorem 5.4.3]{Di},  although formulated and proved in different terms.
As Dimca observed in \cite{Di0}, we also point out here that the relation between vanishing homology and absolute homology is encoded by the morphism $\Phi_{n}$, which is difficult to identify from the equation of $V$. We send the reader to Proposition \ref{p:dimca1} for our extension of this result in case $\dim \Sing V =1$.


\section{Local theory of 1-dimensional singular locus} \label{ss:localtheory}

We shall need several facts from the local theory of singularities with a 1-dimensional singular set. We recall them here, following \cite{Si3}, see also the survey \cite{Si-Cam}.

 We consider a holomorphic function germ $f : (\bC^{n+1},0) \to (\bC,0)$ with singular locus $\Sigma$ of dimension 1. Let  $\Sigma = \Sigma_1 \cup \ldots \cup \Sigma_r $ be the decomposition into irreducible curve components. Let $F$ be the local Milnor fibre of $f$. 
 The homology $\tilde{H}_*(F)$ is
concentrated in dimensions $n-1$ and $n$, namely $H_n (F) = \bZ^{\mu_n}$, 
which is free,  and
$H_{n-1} (F)$ which can have torsion. 

There is a well-defined local system on $\Sigma_i \m \{ 0\}$ having as fibre the homology of the transversal Milnor fibre $\tilde{H}_{n-1} (\tF_i)$, i.e., $\tF_i$ is the Milnor fibre of the restriction of $f$ to the transversal hyperplane section at some $x \in \Sigma_i \m \{ 0\}$, which is an isolated singularity
whose equisingularity class is independent of the point $x$.  
Thus  $\tilde{H}_{*} (\tF_i)$
is concentrated in dimension $n-1$. On this group there acts the \emph{local system monodromy} (also called \emph{vertical monodromy}):

\[ A_i: \tilde{H}_{n-1} (\tF_i) \rightarrow \tilde{H}_{n-1} (\tF_i). \]

As explained in \cite{Si3}, one considers a tubular neighborhood $\cN := \bigsqcup_{i=}^r \cN_i$ of the link $\Sigma\cap S_{\e}^{2n+1}$ of $\Sigma$ in 
$S_{\e}^{2n+1}$ and decomposes the boundary $\partial F$ of the Milnor fibre as
$\partial F = \partial_1 F\cup \partial_2 F$, such that $\partial_2 F := \partial F \cap \cN$ and that  $\partial_1 F\cap \partial_2 F$ retracts to the boundary $\partial\cN$. Then $\partial_2 F = \displaystyle{\mathop{\sqcup}^{r}_{i=1}} \partial_2 F_i$, where $\partial_2 F_i := \partial F\cap \cN_i$.

The homology groups of $\partial_2 F$  are related to the local system monodromies $A_i$ in the following way.   Each boundary component $\partial_2 F_i$ is fibered over the link of $\Sigma_i$ with fibre $\tF_i$.
The Wang sequence of this fibration yields the following non-trivial part, for $n\ge 3$:
\begin{equation}\label{eq:wang}
  0 \rightarrow H_n (\partial_2 F_i) \rightarrow H_{n-1} (\tF_i)
\stackrel{A_i -I}{\rightarrow} H_{n-1} (\tF_i) \rightarrow H_{n-1}
(\partial_2 F_i) \rightarrow 0
\end{equation}

\noindent
In this sequence the following two homology groups play a crucial role:
$ H_n (\partial_2 F) =  \displaystyle{\mathop{\bigoplus}^{r}_{i=1}}
{\rm Ker} (A_i - I)$ and $H_{n-1} (\partial_2 F)  \cong 
\displaystyle{\mathop{\bigoplus}^{r}_{i=1}} \; {\rm Coker} \; (A_i
- I)$.
The first group is free, the second one can have torsion, and they are isomorphic up to torsion. For $n=2$ there is an adapted interpretation of this sequence, cf \cite[Section 6]{Si3}.

What we  will actually need in the following is a relative version of this Wang sequence. Let $\tE_i$ be the transversal Milnor neighborhood containing the transversal fibre $\tF_i$; it is homeomorphic to a $2n$-ball and hence contractible. Let $\partial_2 E_i$  denote the union of such transversal Milnor neighbourhoods along the link $\Sigma_i \cap S_{\e}^{2n+1}$; this may be identified with the tubular neighborhood $\cN_i$, which retracts to the link of $\Sigma_i$. We then have:
\begin{lemma} \label{l:localrelative}
For  $n \geq 2$ 
\[ 0 \rightarrow H_{n+1} (\partial_2 E_i,\partial_2 F_i) \rightarrow H_{n} (\tE_i,\tF_i)
\stackrel{A_i -I}{\rightarrow} H_{n} (\tE_i,\tF_i) \rightarrow H_{n}
(\partial_2 E_i, \partial_2 F_i) \rightarrow 0 \]
is an exact sequence,
and 
\begin{eqnarray*}
 H_{n+1} (\partial_2 E, \partial_2 F) & =  & \displaystyle{\mathop{\bigoplus}^{r}_{i=1}}
{\rm Ker} (A_i - I) \\
 H_{n}( \partial_2 E, \partial_2 F) & \cong &
\displaystyle{\mathop{\bigoplus}^{r}_{i=1}} \; {\rm Coker} \; (A_i- I) 
\end{eqnarray*}
\end{lemma}
\begin{proof}
For $n > 2$ the statement follows immediately from the above Wang sequence \eqref{eq:wang} and the definitions of $\tE_i$ and $\partial_2 E_i$. One observes that $n=2$ is no longer a special case like it was in the absolute setting (see the remark after \eqref{eq:wang}).
\end{proof} 

 The non-trivial part of the  long exact sequence of the pair $(F,\partial_2 F)$ is the following 
6-terms piece. More precisely, we need the following result:

\begin{proposition}\cite{Si3} \label{t:varladder}
The sequence
\[ 0 \rightarrow H_{n+1}(F,\partial_2 F) \rightarrow H_{n}(\partial_2 F)
\rightarrow H_n (F) 
 \rightarrow H_{n}(F,\partial_2 F) \rightarrow H_{n-1}(\partial_2 F)
\rightarrow H_{n-1} (F)
\rightarrow 0 \]
is exact. Moreover \[ H_{n+1}(F,\partial_2 F) \cong H_{n-1} (F)^{\mbox{\tiny free}} \; \mbox{and} \;
H_{n}(F,\partial_2 F) \cong H_{n} (F)   \oplus  H_{n-1} (F)^{\mbox{\tiny torsion}}.\]
\fin
\end{proposition}


\section{The vanishing neighbourhood of the projective hypersurface}\label{ss:1-dim}

We give here the necessary constructions and lemmas that we shall use in the proof of the announced vanishing theorem:

\begin{theorem}\label{t:main}  
If $\dim \Sing V \le 1$ then $H^{\curlyvee}_{j}(V) = 0$ for all $j\not= n+1, n+2$.
\end{theorem}

Let $V := \{ f= 0 \} \subset \bP^{n+1}$ denote a hypersurface of degree $d$ with singular locus $\hat \Sigma$ of dimension one, more precisely $\hat \Sigma$ consists of a union $\Sigma := \cup_i \Sigma_i \cup R$ of irreducible projective curves $\Sigma_i$ and of a finite set of points $R$. 

 We recall that we have denoted by $A = \{ f=h_d=0\}$ the axis of the pencil $\pi : \bV_\Delta \to \Delta$ defined in the Introduction.
One considers the \emph{polar locus}\footnote{the polar locus of a map $(h,f)$ is defined as $\overline{\Sing(h,f) \m (\Sing h \cup \Sing f)}$} of the map $(h_d, f) : \bC^{n+2} \to \bC^2$ and since this is a homogeneous set one takes its image in $\bP^{n+1}$ which will be denoted by $\Gamma(h_d, f)$. 
Let us recall from \cite{ST-bettimax} the meaning of ``general'' for $h_d$ in this setting.
By using the Veronese embedding of degree $d$ we find a Zariski open set $\cO$ of linear functions in the target such that whenever $g\in \cO$ then its pull-back is a general homogeneous polynomial $h_d$ defining a hypersurface $H := \{ h_d = 0\}$ which is transversal to $V$ in the stratified sense, i.e. after endowing $V$ with some Whitney stratification, of which the strata are as follows:  the isolated singular points  $\{ \{ r \}, r \in R\}$ of $V$ and the point-strata $\{ \{ q \}, q \in Q\}$ in $\Sigma$, the components of $\Sigma \m Q$ and the open stratum $V\m \hat\Sigma$. Such $h_d$ will be called \textit{general}. This definition implies that $A$ intersects $\hat\Sigma$ at general points, in particular does not contain any points of $Q \cup R$.  It was shown in \cite[Lemma 5.1]{ST-bettimax} that the space $\bV_\Delta$ has isolated singularities: $\Sing \bV_\Delta = (A \cap \Sigma) \times \{ 0\}$, 
and  that $\pi : \bV_\Delta \to \Delta$ is a map with 1-dimensional singular locus $\Sing (\pi) = \hat \Sigma \times \{ 0\}$. One of the key preliminary results is the following supplement to \cite[Lemma 5.2]{ST-bettimax}, which extends the proof in \emph{loc.cit.} from Euler characteristic to homology\footnote{A related result was obtained in \cite{PP}. Like in case of \cite{PP}, the proof actually works for any singular locus $\Sing V$ and any general pencil.}:
\begin{lemma}\label{l:B_i}   
If $h_d$ is general then $\Gamma_p(h_d, f)= \emptyset$ at any point $p\in A\times \{ 0\}$. In particular, for a small enough ball $B_p$ centered at $p$, the local relative homology is trivial, i.e.:
\[ H_*(B_p, B_p\cap V_\e) = 0.  \]
\end{lemma}
\begin{proof}
 The fist claim has been proved in \cite[(12)]{ST-bettimax}. Let us show here the second one.
The notation $B_p$ stands for the intersection of  $\bV_\Delta$ with a small ball in some chosen affine chart $\bC^{n+1}\times \Delta$ of the ambient space $\bP^{n+1}\times \Delta$. In particular $B_p$ is of dimension $n+1$.
 Consider the map $(\pi, h_d) :  B_p \to \Delta \times \Delta'$. Consider the germ of the polar locus of this map at $p$, denoted by $\Gamma(\pi, \hat h_d)$, where $\hat h_d$ is the de-homogenization of $h_d$ in the chosen chart.
It follows from the definition of the polar locus that some point $(x,\e) \in \bV_\Delta$, where $\e = - f(x)/h_d(x)$, is contained in  $\Gamma(\pi, h_d) \setminus (\{ f=0\} \cup \{ h_d=0\})$ if and only if $x \in \Gamma(f, h_d) \setminus (\{ f=0\} \cup \{ h_d=0\})$.
By the first statement, $\Gamma(f, h_d)$ is empty at $p$.
The absence of the polar locus implies that $B_p\cap V_\e$ is homotopy equivalent (by deformation retraction) to the space $B_p\cap V_\e \cap \{ h_d=0\}$. The latter is the slice by $\e =$ constant of the space $\bV_\Delta \cap \{ h_d=0\} = \{ f=0\} \times \Delta$, which is a product space. Since this is homeomorphic to the complex link of this space and a product space has contractible complex link, we deduce that $B_p\cap V_\e$ is contractible too. Since $B_p$ is contractible itself, we get our claim.
\end{proof}

\subsection{Notations}\label{ss:notations}
 Let us assume for the moment that $\Sigma$ is irreducible and discuss the reducible case at the end in \S \ref{ss:reducible}. Let $g$ be its genus, in the sense of the definition given at \S \ref{ss:cw2}.
We use the following notations: \\
$P := A\cap \Sigma$ the set of axis points of $\Sigma$;  $Q:=$  the set of special points on $\Sigma$; \\
 $R:=$  the set of isolated singular points. 
\\
 $\Sigma^* := \Sigma \m (P \cup Q)$ and
 $\cY :=$ small enough  tubular neighborhood of $\Sigma^*$ in $\bV_\Delta$.\\
$B_p, B_q, B_r$ are small enough Milnor balls within $\bV_\Delta \subset \bP^{n+1}\times \Delta$  at the points $p \in P, q \in Q, r \in R$ respectively, and\\
$B_P := \bigsqcup_p B_p$, $B_Q := \bigsqcup_q B_q$ and $B_R := \bigsqcup_r B_r$.\\
Let us denote the projection of the tubular neighborhood by $\pi_\Sigma : \cY \to \Sigma^*$.

\medskip
Let  $\nu := \# P$ be the number of axis points. 
 At any special point $q\in Q$, let $S_q$ be the index set of locally irreducible branches of the germ $(\Sigma,q)$, and let $\gamma := \sum_{q\in Q}\# S_q$.

By homotopy retraction and by excision we have:
\begin{equation}\label{eq:directsumdecomp}
  H_*(\bV_\Delta, V_\e) \simeq H_*(\cY \cup B_P \cup B_Q,  V_\e \cap \cY \cup  B_P \cup  B_Q) \bigoplus \oplus_{r\in R} H_*(B_r,V_\e \cap B_r).
\end{equation}

\medskip
\noindent
We introduce the following shorter notations:\\
\[ \cX := B_P \sqcup  B_Q \; , \; \cA := V_\e \cap \cX  \; , \;  
 \cB :=  V_\e \cap \cY \; , \;
\cZ := \cX\cap \cY \; , \; \cC := \cA\cap \cB \]
 \[ (\cX_p,\cA_p) := (B_p,V_\e \cap B_p) \; , \;
(\cX_q,\cA_q) := (B_q,V_\e \cap B_q) \; . \]

\smallskip

 In the new notations, the first direct summand of \eqref{eq:directsumdecomp} is
$H_*(\cX\cup \cY, \cA\cup \cB)$, thus \eqref{eq:directsumdecomp} writes as follows:
\begin{equation}\label{eq:directsumdecomp2}
  H_*(\bV_\Delta, V_\e) \simeq H_*(\cX\cup \cY, \cA\cup \cB) \bigoplus \oplus_{r\in R} H_*(B_r,V_\e \cap B_r).
\end{equation}

Note that each direct summand $H_*(B_r,V_\e \cap B_r)$ is concentrated in dimension $n+1$ since it identifies to the Milnor lattice of the isolated singularities germs $(V_0, r)$, where $\mu_r$ denotes its Milnor number.
This aspect was treated in \S \ref{s:intro} in case of isolated singularities. We shall therefore deal from now on with the first term in the direct sum of \eqref{eq:directsumdecomp}. 

We next consider the relative Mayer-Vietoris long exact sequence:
\begin{equation}\label{eq:mv}
 \cdots \to H_{*}(\cZ,\cC) \to H_{*}(\cX,\cA) \oplus H_{*}(\cY,\cB) \to  H_{*}(\cX\cup \cY, \cA\cup \cB) \stackrel{\partial_s}{\to} \cdots
\end{equation}
 of the pair $(\cX\cup \cY, \cA\cup \cB)$ and we compute in the following each term of it.

\subsection{The homology of $(\cX,\cA)$}\label{ss:term}
One has the direct sum decomposition $H_{*}(\cX,\cA) \simeq \oplus_p H_*(\cX_p,\cA_p) \bigoplus  \oplus_q H_*(\cX_q,\cA_q)$ 
since $\cX$ is a disjoint union. The triviality $H_*(\cX_p,\cA_p)= 0$ follows by Lemma \ref{l:B_i}.
The pairs $(\cX_q,\cA_q)$ are local Milnor data of the germs $(V, q)$ with 1-dimensional singular locus and therefore the relative homology $H_*(\cX_q,\cA_q)$ is concentrated in dimensions $n$ and $n+1$.

\subsection{The homology of $(\cZ,\cC)$}\label{ss:term2}
The pair $(\cZ,\cC)$ is a disjoint union of pairs localized at points $p \in P $ and $q \in Q$.
For axis points $p \in P$ we have a unique pair $(\cZ_p, \cC_p)$ as bundle over the link of $\Sigma$ at $p$ with fibre the transversal data $(\tE_p, \tF_p)$, in the notations of \S \ref{ss:localtheory}. For the non-axis points $q\in Q$ we have one contribution for each {\it locally irreducible branch of the germ $(\Sigma,q)$}. Let  $S_q$ be the index set of all these branches at $q\in Q$. We get the following decomposition:
\begin{equation}\label{eq:decomp}
  H_{*}(\cZ,\cC) \simeq \oplus_{p\in P} H_{*}(\cZ_p, \cC_p) \bigoplus \oplus_{q\in Q} 
\oplus_{s\in S_q} H_{*}(\cZ_s, \cC_s).
\end{equation}

More precisely, one such local pair $(\cZ_s, \cC_s)$ is the bundle over the corresponding component of the link of the curve germ $\Sigma$ at $q$ having as fibre the local transversal Milnor data $(\tE_s, \tF_s)$.  
In the notations of \S \ref{ss:localtheory}, we thus have:
 $\partial_2 \cA_q =\mathop{\bigsqcup}_{s \in S_q} \cC_s $.

The relative homology groups in the above decomposition \eqref{eq:decomp} depend on the \emph{vertical monodromy} via the Wang sequence of Lemma \ref{l:localrelative}, as follows:
\begin{equation}\label{eq:wang2}
 0 \rightarrow H_{n+1} (\cZ_s,\cC_s)) \rightarrow H_{n} (\tE,\tF)
\stackrel{A_s -I}{\rightarrow} H_{n} (\tE,\tF) \rightarrow H_{n}
(\cZ_s,\cC_s) \rightarrow 0.
\end{equation}
 
Note that here the transversal data is independent of the points $q$ or the index $s$ since
$\Sigma^*$ is connected and therefore the transversal fibre is uniquely defined. However the vertical monodromies $A_s$ depend of $s\in S_q$.
From the above and from Lemma \ref{l:B_i} we get:
\begin{lemma}\label{l:concentration}
At points $q \in Q$,  for each $s\in S_q$ one has:
\begin{eqnarray*}
  H_{k}(\cZ_s,\cC_s ) = 0 &  k\not= n, n+1,\\
H_{n+1}(\cZ_s,\cC_s) \cong {\ker} \; (A_s- I), &  H_{n}(\cZ_s,\cC_s ) \cong {\coker} \; (A_s- I).
\end{eqnarray*}

At axis points $p \in P$ and more generally, at any point $p$ such that $A_p = I$, one has:
\begin{eqnarray*}
 H_{k}(\cZ_p, \cC_p) = 0 &  k\not= n, n+1,\\
H_{n+1}(\cZ_p, \cC_p) \cong H_{n}(\cZ_p, \cC_p) &  \cong H_n(\tE,\tF) = \bZ^{\tmu}.
\end{eqnarray*}
\end{lemma}

\begin{proof}
The first statement follows from the Wang sequence \eqref{eq:wang2} and since $H_{k} (\tE,\tF)$ is concentrated in $k=n$. The last statement follows because 
 the axis points $p\in P$ are general points of $\Sigma$ and therefore  the local vertical monodromy $A_p$ is the identity.
\end{proof}

We conclude that $H_*(\cZ,\cC)$ is concentrated in dimensions $n$ and $n+1$ only.
\smallskip

\subsection{The CW-complex structure of $(\cZ,\cC)$} \label{ss:cw}
The pair $(\cZ_s,\cC_s)$ has moreover the following structure of a relative CW-complex, up to homotopy.
Each bundle over some circle link can be obtained from a trivial bundle over an interval by identifying the fibres above the end points via the geometric vertical monodromy $A_s$.
In order to obtain $\cZ_s$ from $\cC_s$ one can start by first attaching $n$-cells  $c_{1}, \ldots, c_{\tmu}$ to the fibre $\tF$ in order to kill the $\tmu$ generators of $H_{n-1}(\tF)$ at the identified ends,  and next by attaching $(n+1)$-cells $e_{1}, \ldots, e_{\tmu}$ to the preceding  $n$-skeleton. The attaching of some $(n+1)$-cell is as follows: consider some $n$-cell $a$ of the $n$-skeleton and take the cylinder $I\times a$ as an $(n+1)$-cell. Fix an orientation of the   
circle link, attach the base $\{ 0\}\times a$ over $a$, then follow the circle bundle in the fixed orientation by the monodromy $A_s$ and attach the end $\{ 1\}\times a$ over $A_s(a)$.
At the level of the cell complex, the boundary map of this attaching identifies to $A_s-I : \bZ^{\tmu} \to \bZ^{ \tmu}$.

\subsection{The CW-complex structure of $(\cY,\cB)$} \label{ss:cw2}

 For technical reasons we introduce one more puncture on $\Sigma$.  Let us therefore define the total set of punctures $T := P \sqcup Q \sqcup \{ y\}$, where $y$ is a general point of $\Sigma$, then redefine $\Sigma^* := \Sigma\m T$ by considering the new puncture $y$. 

 Let $n : \tilde \Sigma \to \Sigma$ be the normalization map. Then we have the isomorphism $\Sigma^* = \Sigma \m T \simeq \tilde \Sigma \m n^{-1}(T)$.  We choose generators of $\pi_1(\Sigma^*, z)$ for some base point $z\in \Sigma^*$ as follows: first the $2g$ loops (called \emph{genus loops} in the following) which are generators of  $\pi_1(\tilde\Sigma, n^{-1}(z))$, where $g$ denotes the genus of the normalization $\tilde \Sigma$, and next by choosing one loop for each puncture of $P$ and of $Q$. The total set of loops is indexed by the set $T' = T \m \{ y\}$.  
Let us denote by $W$ the set of indices for the union of $T'$ with the genus loops, and therefore $\# W = 2g + \nu +\gamma$, where $\nu := \# P$ and $\gamma := \sum_{q\in Q}\# S_q$ (recalling the Notations \S\ref{ss:notations}).
By enlarging ``the hole'' defined by the puncture $y$, we retract $\Sigma^*$ to the chosen bouquet configuration of non-intersecting loops, denoted by $\Gamma$. The number of loops is $2g + \nu + \gamma $.
Note that $\nu >0$ since there must be at least $d$ ``axis points''.

 The pair $(\cY,\cB)$ is then homotopy equivalent (by retraction) to the pair  $(\pi_\Sigma^{-1}(\Gamma),\cB\cap \pi_\Sigma^{-1}(\Gamma))$. We endow the latter with the structure of a relative CW-complex as we did with $(\cZ,\cC)$  at \S \ref{ss:cw}, namely for each loop the similar CW-complex structure as we have defined above for some pair $(\cZ_s,\cC_s)$, see Figure \ref{f:1}.
\begin{figure}[hbtp]
\begin{center}
\epsfxsize=5cm  
\leavevmode
\resizebox{!}{2.3in}{\includegraphics{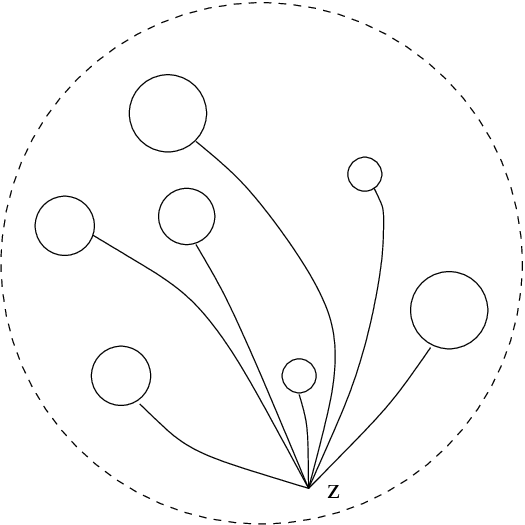}} 
\end{center}
\caption{{\em Retraction of the surface $\Sigma^*$
 }}
\label{f:1}
\end{figure}
The difference is that the pairs $(\cZ_s,\cC_s)$ are disjoint whereas in $\Sigma^*$ 
the loops meet at a single point $z$. We thus  take as reference the transversal fibre $\tF  = \cB \cap\pi_\Sigma^{-1}(z)$ above the point $z$, namely we attach the $n$-cells (thimbles) only once to this single fibre in order to kill the $\tmu$ generators of $H_{n-1}(\tF)$. 
The $(n+1)$-cells of $(\cY,\cB)$ correspond to the fibre bundles over the loops in the bouquet model of $\Sigma^*$. Over each loop,  one attaches a number $\tmu$ of $(n+1)$-cells to the fixed $n$-skeleton described before, more precisely one $(n+1)$-cell over one $n$-cell generator of the $n$-skeleton.
We extend the notation $(\cZ_j,\cC_j)$ to genus loops, although they are not contained in $(\cZ,\cC)$.

 The attaching map of the $(n+1)$-cells corresponding to the bundle over some loop
  can be identified with $A_j -I : \bZ^{\tmu} \to \bZ^{\tmu}$,
where the local system monodromies $A_j$ corresponding to loops may not be local monodromies, and where $\bZ^{\tmu}$ is the homology group $H_{n-1}(\tF)$ of the transversal fibre over $z$ and hence the same for each loop.

From this CW-complex structure we get the following precise description in terms of the local monodromies of the transversal local system:   
\begin{lemma}\label{p:concentration}
\[\begin{array}{l}
 H_{k}(\cY,\cB) = 0   \mbox{ if } k \ne n, n+1,   \\
H_{n}(\cY,\cB) \simeq  \bZ^{\tmu} /  \langle \im (A_j-I) \mid  j\in W \rangle ,\\
 H_{n+1}(\cY,\cB) \mbox{ is free of rank } (2g + \nu +\gamma -1) \tmu + \rank H_{n}(\cY,\cB) \le (2g + \nu +\gamma ) \tmu,\\
 H_{n+1}(\cY,\cB) \mbox{ naturally contains } 
 \bigoplus_{j \in W }H_{n+1}(\cZ_j, \cC_j)\mbox{ as a direct summand, } \\
\chi(\cY,\cB) = (-1)^{n-1}(2g +\nu + \gamma -1) \tmu.

\end{array}\]
\end{lemma}
\begin{proof}
The relative CW-complex model of $(\cY,\cB)$ contains only cells in dimension $n$ and $n+1$. 
At the level $n+1$, the chain group is generated 
by all $(n+1)$-cells corresponding to elements of $W$.  Then $H_{n+1}(\cY,\cB)$  identifies to the kernel of the boundary map $\partial$ in the second row of the following commuting diagram of exact sequences (provided by Lemma \ref{l:localrelative} and by \eqref{eq:wang2}), where the vertical arrows are induced by inclusion: 

\begin{equation}\label{eq:commut}
 \begin{array}{cccccccccccc}
 0 & \to & H_{n+1}(\cZ_j, \cC_j) & \hookrightarrow &  H_{n}(\tE_j, \tF_j) & \stackrel{\partial_j}{\rightarrow} & H_{n}(\tE_j, \tF_j) & \rightarrow & H_{n}(\cZ_j, \cC_j) & \to & 0   \\ 
 \ & \ &  \downarrow & \ & \downarrow & \ & \downarrow = & \ & \downarrow & \ & \  \\
0 & \to & H_{n+1}(\cY,\cB) & \hookrightarrow &  \oplus_{j\in W}H_{n}(\tE_j, \tF_j) & \stackrel{\sum_{j\in W}\partial_j}{\rightarrow} & H_{n}(\tE_j, \tF_j) & \rightarrow & H_{n}(\cY,\cB) & \to & 0

\end{array}
\end{equation}
For any $j\in W$ we get that the first vertical arrow is injective.
By taking the direct sum over $j\in W$ in the left hand commutative square of \eqref{eq:commut}   we get an injective map $\bigoplus_{j \in W}H_{n+1}(\cZ_j, \cC_j) \hookrightarrow H_{n+1}(\cY,\cB)$.
It follows that the image is a direct summand.

Counting the ranks in the lower exact sequence yields the above claimed formula for $\chi$.

\end{proof}


\section{Concentration of the vanishing homology. Proof of Theorem \ref{t:main}}\label{ss:proofmain}
Lemma \ref{l:concentration}, \S \ref{ss:term} and Lemma \ref{p:concentration}
show that the  terms $H_{*}(\cX ,\cA)$, $H_{*}(\cY ,\cB)$ and $H_{*}(\cZ ,\cC)$ of the Mayer-Vietoris sequence \eqref{eq:mv} are concentrated in dimensions $n$ and $n+1$ only, which fact implies the following result:

\begin{proposition}\label{p:7-term}
The relative Mayer-Vietoris sequence \eqref{eq:mv} is trivial except for the following  7-terms sequence:
\begin{equation}\label{eq:MV7}\begin{array}{l}
 0 \to H_{n+2}(\cX \cup \cY , \cA \cup \cB) \to \ \ \ \ \  \\
 \ \ \ \ \ \to H_{n+1}(\cZ,\cC) \to H_{n+1}(\cX,\cA) 
\oplus H_{n+1}(\cY,\cB) \to  H_{n+1}(\cX \cup \cY , \cA \cup \cB) \to \\
 \ \ \ \ \ \to  H_{n}(\cZ,\cC)\stackrel{j}{\to} H_{n}(\cX,\cA) \oplus H_{n}(\cY,\cB)
 \to H_{n}(\cX \cup \cY , \cA \cup \cB) \to  0.
 \end{array}
\end{equation}\fin
\end{proposition}
From Proposition \ref{p:7-term} and from \eqref{eq:directsumdecomp2} it follows that the vanishing homology $H_*(\bV_\Delta, V_\e)$ is concentrated in dimensions $n, n+1, n+2$.

We pursue by showing that $H_n(\bV_\Delta, V_\e)=0$, i.e. that the last term of \eqref{eq:MV7} is zero.
We  need the relative version of the exact sequence of  Proposition \ref{t:varladder}, which appears to have an important overlap with our relative Mayer-Vietoris sequence. 

\begin{proposition} \label{p:varladderRel}
For any $q \in Q$, the sequence

\[ \begin{array}{l}   
0 \rightarrow H_{n+1}(\cA_q,\partial_2 \cA_q) \rightarrow \mathop{\bigoplus} _{s \in S_q} H_{n+1}(\cZ_s,\cC_s)
\rightarrow H_{n+1} (\cX_q,\cA_q) \rightarrow  \ \ \ \ \ \\
 \ \ \ \ \  \rightarrow H_{n}(\cA_q,\partial_2 \cA_q) \rightarrow \  \mathop{\bigoplus}_{s \in S_q} H_{n}(\cZ_s,\cC_s) \ 
\rightarrow  \ H_{n} (\cX_q,\cA_q)  \ 
\rightarrow  \ 0
   \end{array} \]

is exact for $n\ge 2$. Moreover we have: 
\[ H_{n+1}(\cA_q,\partial_2 \cA_q) \cong H_{n-1} (\cA_q)^{\mbox{\tiny free}} \; \mbox{and} \;
H_{n}(\cA_q,\partial_2 \cA_q) \cong H_{n} (\cA_q)   \oplus  H_{n-1} (\cA_q)^{\mbox{\tiny torsion}} \]
\end{proposition}
\begin{proof}
Note that we have the following coincidence of objects which have different notations 
 in the projective setting of this section and in the local setting of \S \ref{ss:localtheory}:\\
$\cA_q := F$, $\partial_2 \cA_q :=\partial_2 F$.

We also have the isomorphisms $H_{*+1} (\cX_q,\cA_q) = \tilde{H}_{*}(\cA_q)$ since $\cX_q$ is contractible, 
then  $H_{*}(\partial_2 \cA_q) = \mathop{\bigoplus}_{s\in S_q} H_{*}(\cC_s)$ by definition, and
$H_{k}(\cC_s) = H_{k+1}(\cZ_s,\cC_s)$ for $k>2$, since $\cZ_s$ contracts to a circle.
We use Proposition \ref{t:varladder} and  check that, like in Lemma \ref{l:localrelative} on another (but similar) relative situation, the case
$n=2$ does not give any problem for the exactness of the above sequence.
\end{proof}

\subsection{Surjectivity of $j$} \label{ss:surj}

We focus on the following map which occurs in the 7-term exact sequence \eqref{eq:MV7}:
\begin{equation} \label{eq:j}
  j = j_1 \oplus j_2: H_{n}(\cZ,\cC) \to H_{n}(\cX,\cA) \oplus H_{n}(\cY,\cB).
\end{equation}


\subsubsection{The first component $j_1: H_{n}(\cZ,\cC) \to H_{n}(\cX,\cA)$} \label{sss:first} \ \\
Note that, as shown above, we have the following direct sum decompositions of the source and the target:
\[ \begin{array}{l}
  H_{n}(\cZ,\cC) = \oplus_{p \in P}H_{n}(\cZ_p,\cC_p) \bigoplus\oplus_{q\in Q} \oplus_{s\in S_q}  H_{n}(\cZ_s,\cC_s) \oplus H_n(\cZ_y,\cC_y)                 , \\
  H_{n}(\cX,\cA) = \oplus_{q\in Q} H_n(\cX_q,\cA_q)  \bigoplus H_n(\cX_y,\cA_y)  .
\end{array}
\]
The
 terms corresponding to the points $p\in P$  are mapped by $j_1$ to zero since $H_{n}(\cX_p,\cA_p)=0$ by Lemma \ref{l:B_i}. 
 Next, as shown in Proposition \ref{p:varladderRel}, at the special points $q\in Q$ we have surjections:
$\mathop{\bigoplus}_{s \in S_q} H_{n}(\cZ_s,\cC_s) \rightarrow H_{n} (\cX_q,\cA_q)$ and moreover
$H_n(\cZ_y,\cC_y) \rightarrow H_n(\cX_y,\cA_y)$ is an isomorphism.
This shows that the morphism $j_1$ is surjective.

\subsubsection{The second component $j_2: H_{n}(\cZ,\cC) \to  H_{n}(\cY,\cB)$}  \ \\
Both sides are described with a relative CW-complex as explained in \S \ref{ss:cw2}. At the level of $n$-cells there are $\tmu$  $n$-cell generators for each  $p \in P$, and the same for each $s \in S_q$ and any $q\in Q$. Each of these generators is mapped bijectively to the single cluster of $n$-cell generators attached to the reference fibre $\tF$ (which is the fibre above the common point of the loops, see also Figure 1).
 We have the same boundary map for each axis point $p\in P$ in the source and in the target of $j_2$ and therefore, at the level of the $n$-homology, the restriction ${j_2}_| :  H_{n}(\cZ_p,\cC_p) 
	\to  H_{n}(\cY,\cB)$ is surjective. 
	Since we have at least one axis point on $\Sigma$ and $ \oplus_{p\in P}H_n(\cZ_p,C_p) \subset \ker j_1$, this shows that  the restriction ${j_2}_| : \oplus_{p \in P} H_{n}(\cZ_p,\cC_p) \to  H_{n}(\cY,\cB)$ is surjective too.
We have thus proven the surjectivity of $j$ and
in particular the following statement:
\begin{proposition}\label{p:6terms}
$H_n(\bV_\Delta, V_\e) = 0$ and in  particular the relative Mayer-Vietoris sequence \eqref{eq:MV7} reduces to the 6-terms sequence:

\begin{equation}\label{eq:MV6}\begin{array}{l}
 0 \to H_{n+2}(\cX \cup \cY , \cA \cup \cB) \to H_{n+1}(\cZ,\cC) \to H_{n+1}(\cX,\cA) \oplus H_{n+1}(\cY,\cB) \\
\  \ \to  \  H_{n+1}(\cX \cup \cY , \cA \cup \cB) \ \to \   H_{n}(\cZ,\cC) \ \stackrel{j}{\to} H_{n}(\cX,\cA) \  \oplus \  H_{n}(\cY,\cB) \ \ \ \to   0
 \end{array}
\end{equation}

 \end{proposition}

This shows that  the relative homology
$H_*(\bV_\Delta, V_\e)$ is concentrated at the levels $n+1$ and $n+2$, and thus finishes the proof of Theorem \ref{t:main} in case of irreducible $\Sigma$.


\subsection{Reducible $\Sigma$} \label{ss:reducible} 
Let  $\Sigma = \Sigma_1 \cup \ldots \cup \Sigma_{\rho}$ be the decomposition into irreducible components.
 The proof of Theorem \ref{t:main} in the reducible case remains the same modulo the following small changes and additional notations:
\begin{enumerate}
 \item For each $i$ one considers the set $Q_i$ of special singular points of $\Sigma_i$. The points of intersection $\Sigma_{i_1} \cap \Sigma_{i_2}$ for $i_1\not= i_2$ are considered as special points of both sets $Q_i$ and $Q_j$, and therefore the union $Q := \cup_i Q_i$ is not disjoint. For some $q\in \Sigma_{i_1} \cap \Sigma_{i_2}$, the set of indices $S_q$ runs over all the local irreducible components of the curve germ $(\Sigma, q)$.
Nevertheless, when we are counting the local irreducible branches at some point $q\in Q_i$ on a specified component $\Sigma_i$ then the set $S_q$ will tacitly mean only those local branches of $\Sigma_i$ at $q$.

\item The pair $(\cY,\cB)$ is a disjoint union and its homology decomposes accordingly, namely $H_{*}(\cY,\cB) = \bigoplus_{1\le i \le \rho}  H_{*}(\cY_i,\cB_i)$.
\item  For each component $\Sigma_i$ one has  its transversal Milnor fibre denoted by $\tF_i$ and its transversal Milnor number $\tmu_i$. \end{enumerate}


\section{Betti numbers of hypersurfaces with 1-dimensional singular locus}\label{s:Betti}

By Theorem \ref{t:main}, the vanishing homology of  a hypersurface $V\subset \bP^{n+1}$ with 1-dimensional singularities is concentrated in dimensions $n+1$ and $n+2$.
We show that its $(n+2)$th vanishing homology group depends on the local data of the special points $Q$ and on the genus loop monodromies along the singular branches. We study this dependence in more detail, we determine the rank of the free group $H_{n+2}(\bV_\Delta, V_\e)$, and discover mild conditions which ensure the vanishing of this group.

We continue to use  the notations of \S\ref{ss:1-dim}. Let us especially recall the notations from \S \ref{ss:cw2} adapted here to the general setting of a reducible singular locus $\Sigma = \cup_{i=1}^{\rho}\Sigma_i$.
For any $1\le i\le \rho$,  $\Sigma_i^* = \Sigma_i \m (P_i\sqcup Q_i\sqcup \{y_i\} )$ retracts to a bouquet $W_i$ of $2g_i + \nu_i + \gamma_i$ circles, where $g_i$ denotes the genus of the normalization $\tilde \Sigma_i$, where $\nu_i := \# P_i$ is the number of axis points $A\cap \Sigma_i$,   where $\gamma_i := \sum_{q\in Q_i}\# S_q$ and $Q_i$ denotes the set of special points of $\Sigma_i$,   the set $S_q$ is indexing the local branches of $\Sigma_i$ at $q$, and where $y_i\in \Sigma_i$ is some point not in the set $P_i \cup Q_i$.
We denote by $G_i$  the set of genus loops of $W_i$.

By Proposition \ref{p:7-term} we have  $H_{n+2}(\bV_\Delta, V_\e) = \ker j = \ker [ j_1 \oplus j_2 ]$,
where 
$$j_1 \oplus j_2: H_{n+1}(\cZ,\cC) \to H_{n+1}(\cX,\cA) \oplus H_{n+1}(\cY,\cB).$$

The main idea in this section is to embed   $H_{n+2}(\bV_\Delta, V_\e)$ into the module
 $\bD = \oplus_{i=1}^\rho \bD_i$, where $\bD_i$ is the image of the diagonal map: 
\[ \Delta_*^i: H_n(E_i^{\pitchfork},F_i^{\pitchfork}) \to \oplus_{q \in Q_i} \oplus_{s \in S_q}  H_n(E_i^{\pitchfork},F_i^{\pitchfork}), \ \ \ a \mapsto (a, a, \ldots , a).\]


The space $\bD$ will serve as a reference space and is isomorphic to 
$\oplus_{i=1}^\rho H_n(F_i^{\pitchfork}) =  \oplus_{i=1}^\rho \bZ^{\mu_i^\pitchfork}.$

The source and the target of $j_1 \oplus j_2$ have a direct sum decomposition at level $n+1$, like has been discussed at \S\ref{ss:surj} for the $n$-th homology groups\footnote{we remind from \S \ref{ss:reducible} that the notation $S_q$ depends of whether the point $q$ is considered in $Q$ or in $Q_i$, namely it takes either the local branches of $\Sigma$ at $q$, or the local branches of $\Sigma_i$ at $q$, accordingly.}:
\begin{equation}\label{eq:directsum}
   j_1 \oplus j_2: \oplus_{p\in P} H_{n+1}(\cZ_p,\cC_p) \oplus_{q\in Q} \oplus_{s\in S_q}  H_{n+1}(\cZ_s,\cC_s) \oplus_{i=1}^\rho H_{n+1}(\cZ_{y_i},\cC_{y_i})
	\to 
\end{equation} 
\[ \oplus_{q\in Q} H_{n+1}(\cX_q,\cA_q) \bigoplus H_{n+1}(\cY,\cB).\]

By Lemma \ref{l:concentration} we have  $H_{n+1}(\cZ_v,\cC_v) = \ker (A_v - I)$, where:
\[ A_v - I : H_n(E_i^{\pitchfork},F_i^{\pitchfork}) \to H_n(E_i^{\pitchfork},F_i^{\pitchfork})\]
is the vertical monodromy at some point
 $v\in P_i$, or $v\in S_q$ and $q\in Q_i$, or $v=y_i$.
The left hand side of \eqref{eq:directsum} consists therefore of local contributions of the form  $\ker (A_v - I)\subset H_{n} (\tE_i,\tF_i)\simeq H_{n-1}(F_i^{\pitchfork}) \simeq \bZ^{\mu_i^{\pitchfork}}$.

We have studied $j_1$ in  \S \ref{sss:first} at the level $n$. 
 For the $(n+1)$th homology groups, the restriction of $j_1$ to the first summand in \eqref{eq:directsum}  is zero since its image is in $\oplus_{p\in P} H_{n+1}(\cX_p,\cA_p)$ which is zero by Lemma \ref{l:B_i}. The image by $j_1$ of $\oplus_{i=1}^\rho H_{n+1}(\cZ_{y_i},\cC_{y_i})$ is also zero since $H_{n+1}(\cX_{y_i},\cA_{y_i})=H_{n}(\cA_{y_i})=0$.
The restriction of $j_1$ to the remaining summand is the direct sum $\oplus_{q\in Q} j_{1,q}$   of the maps:
\[ j_{1,q}:  \mathop{\oplus}_{s \in S_ q} H_{n+1}(\cZ_s,\cC_s)
\rightarrow H_{n+1} (\cX_q,\cA_q).\]

 By Proposition \ref{t:varladder}, the kernel of  $j_{1,q}$ is equal to $H_{n+1}(\cA_q,\partial_2 \cA_q)$, where $\cA_q$ is the local Milnor fibre of the hypersurface germ $(V,q)$, $q\in Q$, and can be identified  to the free part of $H_{n-1}(\cA_q)$. 
Since   $H_{n+1}(\cA_q,\partial_2 \cA_q)$ is contained in $\oplus_{i=1}^\rho \oplus_{Q_i \ni q} \mathop{\oplus}_{s \in S_q} H_{n+1}(\cZ_s,\cC_s)$, it turns out that the intersection $(\oplus_{i=1}^\rho\bD_i) \cap 
\oplus_{q \in Q} H_{n+1}(\cA_q,\partial_2 \cA_q)$ is well defined. 

\smallskip
After these preparation we can state:

\begin{theorem}\label{t:local_fibres} 
 In the above  notations we have:

\[  H^{\curlyvee}_{n+2}(V) = \left( \bD \cap 
\oplus_{q \in Q} H_{n+1}(\cA_q,\partial_2 \cA_q)
  \right) \cap \oplus_{i=1}^\rho \Delta_*^i(\bigcap_{j\in G_i} \ker(A_j - I))
\]
where $A_j : H_n(E_i^{\pitchfork},F_i^{\pitchfork}) \to H_n(E_i^{\pitchfork},F_i^{\pitchfork})$ denotes the monodromy along the loop of $W_i$ indexed by $j\in G_i$.
 
In particular $H^{\curlyvee}_{n+2}(V)$ is free and its rank is bounded as follows\footnote{note that no multiplicities but only transversal types are involved in the rank formula.}:
\[ \rank H^{\curlyvee}_{n+2}(V) \le \sum_{i=1}^\rho 
\min_{s\in S_q, q\in Q_i,j\in G_i} \left\{ \dim \ker(A_s - I), \dim \ker(A_j - I) \right\}
\le \sum_{i=1}^\rho \mu_i^{\pitchfork}.
\]
\end{theorem}

\begin{proof}
In order to handle the map $j_2$, we recall
the relative CW-complex structure of $(\cY,\cB)$ given in \S \ref{ss:cw2}. 
On each component $W_i$ we have identified the set of points $T_i$ which consists of the axis points $P_i$, the special points $Q_i$, and one general point $y_i$. 
The punctured $\Sigma_i^*$ retracts to a configuration $W_i$ of $2g_i + \nu_i + \gamma_i$ loops indexed by the set $W_i$, based at some point $z_i$, where $2g_i$ of them are ``genus loops''  and the other loops are projections by the normalization map $n_i : \tilde \Sigma_i \to \Sigma_i$ of loops around all the punctures of $\tilde \Sigma_i \m n_i^{-1}(P_i\sqcup Q_i)$. Notice that $\# T_i -1 \ge \nu_i >0$.

Let $W := \sqcup_i W_i$.
Consider the spaces $\cY_{W} := \pi_\Sigma^{-1}(W)$ and $\cB_{W} := \cB \cap \cY_{W}$. We have the homotopy equivalence of pairs $(\cY, \cB) \simeq (\cY_{W}, \cB_{W})$ which has been discussed at \S \ref{ss:cw2} 
and use the CW-complex model for $(\cY_W,\cB_W)$. We also have the decomposition $(\cY,\cB) = \sqcup_{i=1}^\rho (\cY_i,\cB_i)$ according to the components $W_i$

In our representation, the map $j_2$ splits into the direct sum of the following maps, for $i\in \{ 1, \ldots , \rho\}$:
\[
j_{2, i} :  \oplus_{p \in P_i}H_{n+1}(\cZ_p,\cC_p) \oplus_{q\in Q_i} \oplus_{s\in S_q}  H_{n+1}(\cZ_s,\cC_s) \oplus H_{n+1}(\cZ_{y_i},\cC_{y_i})   \to  H_{n+1}(\cY_i,\cB_i).
\]
By Lemma \ref{p:concentration}, the map $j_{2, i}$ restricts to an embedding of the direct sum $\oplus_{p \in P_i}H_{n+1}(\cZ_p,\cC_p) \oplus_{q\in Q_i} \oplus_{s\in S_q}  H_{n+1}(\cZ_s,\cC_s)$ into $H_{n+1}(\cY_i,\cB_i)$.
Note that $H_{n+1}(\cZ_v,\cC_v) = \ker (A_v - I) \subset H_{n} (\tE_i,\tF_i)\simeq H_{n-1}(F_i^{\pitchfork})$ for any point $v\in P_i$ or $v\in S_q$ and $q\in Q_i$.
 The kernel $\ker j_{2, i}$ is therefore determined by the relations induced by the image of the remaining direct summand $H_{n+1}(\cZ_{y_i},\cC_{y_i})$ into $H_{n+1}(\cY_i,\cB_i)$.

More precisely, each $(n+1)$-cycle generator $w$ of $H_{n+1}(\cZ_{y_i},\cC_{y_i})\simeq H_{n} (\tE_i,\tF_i)\simeq H_{n-1}(F_i^{\pitchfork})$ induces one single relation. Namely $j_2(w)$ is  a $(n+1)$-cycle above the loop around the point $y_i$, and since this loop is homotopy equivalent to a certain composition of other loops of $W_i$, it follows that $j_2(w)$ is precisely homologous to the corresponding sum of cycles above the loops in $W_i$. Our scope is to find all such sums which contain as terms only elements from the images $j_2(H_{n+1}(\cZ_p,\cC_p))$ for $p\in P_i$ and $j_2(H_{n+1}(\cZ_s,\cC_s))$ for $s\in S_q$ and $q\in Q_i$. We have the following facts: 

\noindent 
1). By Lemma \ref{l:concentration} and \S \ref{ss:cw},  such images are in the kernels of $A-I$ where $A$ is the vertical monodromy of the loop corresponding to $p\in P_i$ or to $s\in S_q$ and $q\in Q_i$.  Therefore  the expression of $j_2(w)$ contains the sum of  
those generators of $j_2(H_{n+1}(\cZ_p,\cC_p))$ and of $j_2(H_{n+1}(\cZ_s,\cC_s))$ which correspond to the same representative $w \in H_{n-1}(F_i^{\pitchfork})$, for any $p\in P_i$ and any $s\in S_q$ and $q\in Q_i$. This implies that 
$w\in \cap_{s\in S_q, q\in Q_i} \ker(A_s - I)$. Note that the points $p\in P_i$ are superfluous in this intersection since  $A_p =I$ for all such points.

\noindent 
2). Let us consider a pair $\gamma_1$ and $\gamma_2$ of genus loops (whenever $g_i >0$) and let us denote by $B_1$ and $B_2$ the local system monodromy along these loops. The relation produced by $j_2(w)$ contains in principle the following relative cycle along the wedge $\gamma_1 \vee \gamma_2$: it starts from
    the representative $a_w \in H_{n-1}(F_i^{\pitchfork})$ of $w$,  moves in the local system along $\gamma_1$ arriving as $B_1(a_w)$ after one loop at the fibre over the base point $z$, next moved along $\gamma_2$ to $B_2B_1(a_w)$, then in the opposite direction along $\gamma_1$ to $B_1^{-1}B_2B_1(a_w)$ and finally in the opposite direction along $\gamma_2$ to $B_2^{-1}B_1^{-1}B_2B_1(a_w)$.
Our condition tells that the relation produced by $j_2(w)$ does not involve $(n+1)$-cycles along the genus loops since $\im j_2\cap \oplus_{j\in G_i}H_{n+1}(\tE_j, \tF) = 0$, by Lemma \ref{p:concentration} and \eqref{eq:commut}.  Therefore the relative cycles along $\gamma_1$ and along $\gamma_2$ must cancel, which fact amounts to the following two pairs of equalities:
\[ \begin{array}{lcl}
    B_1^{-1}B_2B_1(a_w) = a_w     & \mbox{  and  }  &     B_2B_1(a_w) = B_1(a_w), \\
  B_2^{-1}B_1^{-1}B_2B_1(a_w) = B_1(a_w)  & \mbox{  and  }  &   B_1^{-1}B_2B_1(a_w) = B_2B_1(a_w).
   \end{array}
\]
These equalities are  cyclic, thus the 8 above terms appear to be equal. In particular we get $B_1(a_w) = a_w$ and $B_2(a_w) = (a_w)$ for any $w\in \cap_{s\in S_q, q\in Q_i} \ker(A_s - I)$.
 We conclude to the same equalities for any pair of genus loops. 

 Altogether we obtain the following diagonal presentation of $\ker j_{2, i}$:
\[ \begin{array}{r}
 \ker j_{2, i} = \left\{ (a_w, a_w , \ldots , a_w)\in  \oplus_{q\in Q_i} \oplus_{s\in S_q}  H_{n+1}(\cZ_s,\cC_s) \oplus H_{n+1}(\cZ_{y_i},\cC_{y_i}) \mid \ \ \right.  \\  \ \ \ \ \ \ \left.   w \in \bigcap_{s\in S_q, q\in Q_i} \ker(A_s - I) \cap \bigcap_{j\in G_i} \ker(A_j - I) \right\} \subset \bD_i. 
   \end{array}
\]

Since $H_{n+2}(\bV_\Delta, V_\e) \subset \ker j_2 = \oplus_{i=1}^\rho \Delta_*^i \left( 
\cap_{s\in S_q, q\in Q_i} \ker(A_s - I) \cap \cap_{j\in G_i} \ker(A_j - I) \right)$
we get in particular the claimed inequality for the Betti number $b_{n+2}(\bV_\Delta, V_\e)$.
The freeness of $H_{n+2}(\bV_\Delta, V_\e)$ follows from the fact that $\ker j_2$ is free (as the image of the intersection of free $\bZ$-submodules).	

We also obtain  the desired expression of $H_{n+2}(\bV_\Delta, V_\e) = \ker (j_1 \oplus j_2) = \ker j_1 \cap \ker j_2$ by intersecting $\ker j_2$ with the diagonal expression of $\ker j_1$ given just before the statement of Theorem \ref{t:local_fibres}. 
\end{proof}

\begin{remark}\label{r:irreducibleSingV}
\textbf{Irreducible $\Sigma$.}\\
In case $\Sigma$ is irreducible, the equality of Theorem \ref{t:local_fibres} reads:
\[  H^{\curlyvee}_{n+2}(V) \simeq
\bigcap_{q \in Q} H_{n+1}(\cA_q,\partial_2 \cA_q)
   \cap \bigcap_{j\in G} \ker(A_j - I).
\]
In particular, if there are no special points on $\Sigma$ and the monodromy along every the genus loop is the identity, then $H^{\curlyvee}_{n+2}(V)\simeq H_{n-1}(F^{\pitchfork})$. This situation can be seen in the example $V := \{ xy=0\} \subset \bP^3$ for which $H^{\curlyvee}_{4}(V) \simeq \bZ$ and $\rank H^{\curlyvee}_{3}(V) =1$.
\end{remark}

\begin{remark}\label{r:noaxispointscontribution}
\textbf{$(n+1)$th vanishing Betti number.}\\
It appears that $H^{\curlyvee}_{n+2}(V)$ does not depend neither on the axis points, nor on the isolated singular points of $V$. However 
$H^{\curlyvee}_{n+1}(V)$  depends on those elements since the Euler number does, after \cite[Theorem 5.3]{ST-bettimax}:
\begin{equation}\label{eq:euler}
\chi(\bV_\Delta, V_\e) = (-1)^{n+1} \sum_{i=1}^{\rho} (2g_i + \nu_i+ \gamma_i  -2)\mu_i^\pitchfork
 -
 \sum_{q\in Q} (\chi(\cA_q) - 1)+ (-1)^{n+1}\sum_{r\in R} \mu_r.
\end{equation}
\end{remark}

Theorem \ref{t:local_fibres} is useful when we have information about the transversal monodromies, namely about the eigenspaces corresponding to the eigenvalue 1. We immediately derive: 

\begin{corollary}\label{c:annul0}
 If, for every $i\in \{ 1, \ldots , \rho \}$, at least one of the transversal monodromies along the loops $W_i \subset \Sigma_i^*$ has no eigenvalue 1, then  $H^{\curlyvee}_{n+2}(V)= 0$. \fin
\end{corollary}

We may also apply  Theorem \ref{t:local_fibres} when we have enough information about local Milnor fibres of special points, like in the following case (see also Example \ref{ex:red}):

\begin{corollary}\label{c:annul}
 Assume that for any $i\in \{ 1, \ldots , \rho\}$ there is some special point $q_i\in Q$ such that
the $(n-1)^{th}$ homology group of the local Milnor fibre $\cA_{q_i}$ of the hypersurface germ $(V, q_i)$ has rank zero. Then:
 $$H^{\curlyvee}_{n+2}(V) = 0$$
 and the single non-zero vanishing Betti number
$b_{n+1}^{\curlyvee}(V)$ is given by the formula:
\begin{equation}\label{eq:Betti} 
  \rank H^{\curlyvee}_{n+1}(V) =\sum_i (\nu_i+ \gamma_i + 2g_i -2)\tmu_i +(-1)^n \sum_{q \in Q} (\chi(\cA_q)-1)+  \sum_{r \in R} \mu_r.
\end{equation}
\end{corollary}
\begin{proof}
Let $(w_1,\cdots,w_{\rho})$ be an element of the reference space $\oplus_{i=1}^\rho H_n(E_i^{\pitchfork},F_i^{\pitchfork}) \cong \oplus_{i=1}^\rho \bZ^{\mu_i^{\pitchfork}}$. By the diagonal map this corresponds to elements $w_i \in  H_{n+1}(\cZ_s,\cC_s)$ for $s \in S_q$ and $q \in \Sigma_i$. 
By the discussion introducing Theorem \ref{t:local_fibres} the kernel of some component $j_{1,q}:  \mathop{\oplus}_{s \in S_ q} H_{n+1}(\cZ_s,\cC_s)
\rightarrow H_{n+1} (\cX_q,\cA_q)$ is equal to $H_{n+1}(\cA_q,\partial_2 \cA_q)$ which in turn is identified to the free part of $H_{n-1}(\cA_q)=0$. 
The rank zero condition implies that $w_i=0$ for $i$ such that $q \in \Sigma_i$, thus all $w_i$ are zero.

As for the rank of $H_{n+1}(\bV_\Delta, V_\e)$, the formula follows from the Euler characteristic computation \eqref{eq:euler}.
\end{proof}

\begin{remark}
In case of an irreducible singular set $\Sigma$, Corollary \ref{c:annul} tells that one singular point $q\in Q$ with a $(n-1)$th Betti number of the Milnor fibre equal to zero is sufficient for the vanishing of  $H^{\curlyvee}_{n+2}(V)$.
\end{remark}

\section{Computations of Betti numbers}\label{s:remarks}

\subsection{Vanishing Betti numbers}\label{ss:comput}
As direct application of Theorem \ref{t:local_fibres}, we provide explicit computations of the ranks of the vanishing homology of some projective hypersurfaces.

\begin{example} \label{ex:cubic2} [some cubic hypersurfaces]

If $V :=\{  x^2z  + y^2w = 0 \} \subset \bP^3$ then  $\Sing V$ is a projective line and its generic transversal type is $A_1$.
There are three axis points and two special points $q$ with local singularity type $D_{\infty}$. The hypersurface singularity germ $D_{\infty}$ is an \emph{isolated line singularity} in the terminology of \cite{Si1}. Its Milnor fibre $F$ is homotopy equivalent to the sphere $S^2$,  the transversal monodromy is $- \id$. From Corollary \ref{c:annul0} it follows that  $H^{\curlyvee}_{4}(V) \simeq H_1(F) = 0$ and applying Corollary \ref{c:annul}
we get that $\rank H^{\curlyvee}_{3}(V)=5$.

For $V :=\{  x^2z  + y^2w+ t^3= 0 \} \subset \bP^4$,  $\Sing V$ is again a projective line but its generic transversal type is $A_2$, with three axis points and two special points for both of which the local Milnor fibre $F$ is homotopy equivalent to $S^3 \vee S^3$.  Then Corollary \ref{c:annul} yields $H^{\curlyvee}_{5}(V) \simeq H_2(F) = 0$ and  $\rank H^{\curlyvee}_{4}(V)=10$.
This construction can be iterated, for instance
 $V :=\{  x^2z  + y^2w+ t_1^3 + t_2^3= 0 \} \subset \bP^5$ has $H^{\curlyvee}_{6}(V) = 0$ and $\rank H^{\curlyvee}_{5}(V)=20$.
\end{example}

\begin{example} \label{ex:Rnonvoid} [including an isolated singular point]

Let $ V= \{y^2(x+y-1)(x-y+1) + z^4 = 0 \} \subset 
\bP^3$. We have $\Sing V$ is the disjoint union of  $\Sigma$, a projective line $\{ y=z=0 \}$ with transversal type $A_3$ and a point 
$R = \{ (0:1:0:0) \}$ of type $A_3$. There are two special points: 
$Q = \{ (1:0:0:0), (-1:0:0:0) \}$, each of them with Milnor fibre $S^2 \vee S^2 \vee S^2$. It follows that  $H^{\curlyvee}_{4}(V) = 0$ and $\rank H^{\curlyvee}_{3}(V) = 21$.
\end{example}

\begin{example}\label{ex:red} [singular locus with two disjoint curve components]

 Let $V := \{ f= x^2z^2 + x^2w^2 + y^2z^2 + 2y^2w^2 = 0\} \subset \bP^3$, which is defined by an element $f$ of the ideal  $(x,y)^2 \cap (z,w)^2$. Then $\Sing V = \Sigma = \Sigma_1\cup \Sigma_2$, where  $\Sigma_1= \{x=y=0\}$ and $\Sigma_2 = \{z=w=0\}$.
 It turns out that the generic transversal type at both of the line components of the singular locus  is $A_1$ and that there are exactly four $D_{\infty}$-points  on each of these two components. We are in the situation of Corollary  \ref{c:annul}, hence $H^{\curlyvee}_4(V) = 0$ and $\rank H^{\curlyvee}_3(V)=20$.
\end{example}

\subsection{Computation of vanishing homology groups}\label{ss:compvanishgroups}

Using the full details of the proof of Theorem \ref{t:local_fibres},  we may compute not only  the rank of the vanishing homology groups, but in several examples even the vanishing homology group $H^{\curlyvee}_{n+1}(V)$ itself, as follows.

The main ingredient is the map $j^{[k]} = j^{[k]}_1 \oplus  j^{[k]}_2 : H_k(\cZ,\cC) \rightarrow H_k(\cX,\cA) \oplus H_k(\cY,\cB)$, which was denoted  by $j$ in \eqref{eq:j}. Like in \eqref{eq:directsum}, we use the direct sum splitting into axis, special and auxiliary contributions.
$$ 0 \rightarrow \coker j^{[n+1]} \rightarrow H_{n+1}(\bV_\Delta, V_\e) \rightarrow \ker j^{[n]} \rightarrow 0$$ and the strategy will be to work with $j^{[n+1]}$ and $j^{[n]}$ at the level of generators.

\begin{example}\label{ex:rathompp} 
Let $V :=\{  x^2z  + y^3 + xyw = 0 \} \subset \bP^3$. Then  $\Sing V$ is a projective line with generic transversal type $A_1$,  3 axis points, and a single special point $q$ of local singularity type  $J_{2,\infty}$, see \cite{dJ}.  The latter is  an isolated line singularity germ, cf \cite{Si1}, with Milnor fibre $F$ a bouquet of 4 spheres $S^2$ and the transversal monodromy is the identity. By Theorem   \ref{t:local_fibres} and Corollary \ref{c:annul} we get  $H^{\curlyvee}_{4}(V) \simeq H_1(F) = 0$
 and $\rank H^{\curlyvee}_{3}(V) = 6$. We next can show (but skip the details) that there is an isomorphism  $H^{\curlyvee}_{3}(V) \simeq \bZ^6$.
Note that  Dimca \cite{Di-wh} observed that $V$ has the rational homology of $\bP^2$.
\end{example}


\begin{example}\label{ex:xyz}
 $V :=  \{ xyz = 0\} \subset \bP^3$ of degree $d=3$. Then $V$ is reducible with 3 components, $\Sing V$ is the union of 3 projective lines intersecting at a single point $[0:0:0:1]$, and  the transversal type along each of them is $A_1$.  Following the proof of Theorem \ref{t:local_fibres}, we get $\ker j_2^{[3]} = \oplus_{i=1}^3 H_1(\tF_i) \simeq \bZ^3$
and $\ker j_1^{[3]} \simeq H_1(F)$ where $F$ denotes the Milnor fibre of the non-isolated singularity of $V$ at the single special point $[0:0:0:1]$, which is homotopically equivalent to $S^1 \times S^1$. We thus get  $H^{\curlyvee}_{4}(V) \simeq H_1(F) \simeq  \bZ^2$. 

The axis $A$ of our pencil has degree 9 and intersects each of the components of $V$ at 3 general points. Hence $\nu_i = 3$,  $\gamma_i= 1$ for any $i=1,2,3$.

Applying formula \eqref{eq:euler} 
we get that the vanishing Euler characteristic is $-5$, and that $\rank H^{\curlyvee}_{3}(V) =7$. We can moreover show the freeness of this group; we skip the details.
\end{example}


\subsection{The surfaces case}\label{ss:surf}
 
Several examples in the previous subsections are surfaces and the computation of $H_4$ and the rank of $H_3$ could be simplified by counting the number of irreducible components of $\Sigma^*$. Indeed in case of surfaces $V \subset \bP^3$ we have:
\begin{equation}\label{e:components}
H^{\curlyvee}_{4}(V) \simeq \bZ^{r-1},
\end{equation}
where $r$ is the number of irreducible components of $V$.  

 We have included these examples anyhow as applications of our method.

Combining \eqref{e:components} with Theorem \ref{t:local_fibres} yields several consequences on the  singular set and its generic transversal types. We mention here only one:

\begin{corollary}
 $r -1 \le  \sum_{i=1}^\rho \mu_i^{\pitchfork}$.
\fin
\end{corollary}


\subsection{Absolute homology of projective hypersurfaces.} \label{ss:abshom}
If $\dim \Sing V = 1$ then from Theorem \ref{t:main} and the long exact sequence of the pair $(\bV_\Delta, V_\e)$ one gets the isomorphisms:   
$$H_k(V) \simeq H_{k}( V_\e) = H_k(\bP^n) \ \mbox{for} \ k \not= n, n+1,n+2.$$ 
This corresponds to  Kato's result \cite{Ka} in cohomology.\footnote{Dimca states such a result \cite[Theorem 4.1]{Di-how} referring to \cite[p. 144]{Di} for Kato's proof in cohomology \cite{Ka}.}

In the remaining dimensions we have
an 8-term exact sequence:
\begin{equation} \label{eq:8-term}
\begin{array}{l}
0\to H_{n+2}(V_\e) \to H_{n+2}(\bV_\Delta) \to H_{n+2}(\bV_\Delta, V_\e) \xrightarrow{\Phi_{n+1}}
H_{n+1}(V_\e)  \ \ \ \ \ \ \ \ \ \ \ \ \ \ \  \\ \ \ \ \ \ \ \ \ \ \ \ \ \ \ \   \to H_{n+1}(\bV_\Delta) \to H_{n+1}(\bV_\Delta, V_\e) \xrightarrow{\Phi_{n}} H_{n}(V_\e) \to H_{n}(\bV_\Delta) \to 0
\end{array}
\end{equation}
 from which we obtain, with help of Theorem \ref{t:local_fibres}:
\begin{proposition}\label{p:dimca1}
Let $\dim \Sing V = 1$.  If $n$ is even, then:
\begin{enumerate}
\item[(c)] $H_{n+2}(V) \simeq \bZ \oplus  H^{\curlyvee}_{n+2}(V)$,
\item[(d)] $H_{n+1}(V)  \simeq \ker \Phi_{n}$, 
\item[(e)] $H_{n}(V)  \simeq   \coker \Phi_{n}$,
\end{enumerate}
whereas for any $n$ one has the following inequalities:
\begin{enumerate}
\item $ b_{n+2}(V) \le 1+ \sum_{i=1}^\rho \mu_i^{\pitchfork}$
\item $b_n(V) \le \dim H_n(V_\e)$.
\end{enumerate}
\end{proposition}

This can be regarded as a natural extension of Proposition \ref{p:dimca} to 1-dimensional singularities, thus extending also Dimca's corresponding result for isolated singularities that was discussed in \S \ref{s:vh-is}.
Like  in the isolated singularities setting, one has to deal with the difficulty of identifying $\Phi_n$ from the equation of $f$.

\begin{example}

Let $V := \{f(x,y) + f(z,w) = 0 \} \subset \bP^3$,  where $f(x,y)= y^2 \prod_{i=1}^3 (x - \alpha_i y)$, with $\alpha_i\not= 0$ pairwise different.
Its singular set is the smooth line given by $y=w = 0$, with generic transversal type $y^2 + w^3$.
There are two special points $[0:0:1:0]$ and $[1:0:0:0]$, each  with Milnor fibre a bouquet of spheres $S^2$.
 By Corollary \ref{c:annul} we get $H^{\curlyvee}_{4}(V) = 0$ and from the Euler characteristic formula \eqref{eq:euler} (by computing its ingredients) we get $\rank H^{\curlyvee}_{3}(V) = 38$.
One can compute the eigenvalues of the monodromies for all types of singular points; they are all different from 1. By using Randell's criterion \cite[Proposition 3.6]{Ra} one can show that $V$ is a $\bQ$-homology manifold. Since a homology manifold satisfies Poincar\'e duality, it follows e.g. that $H_{3}(V;\bQ) \cong H_{1}(V;\bQ) \cong H_{1}(\bP^n;\bQ) =0$ and  $H_{4}(V;\bQ)\cong H_{0}(V;\bQ) \cong H_{0}(\bP^n;\bQ) \cong \bQ$.  By computations and the exact sequence \eqref{eq:8-term}  we get also $H_{2}(V;\bQ)$ since $\rank H_{2}(V) = \rank H_{2}(V_\e) - \rank  H_{3}(\bV_\Delta, V_\e) = 53-38= 15$.

\end{example}


%

\end{document}